\documentclass[journal]{IEEEtran}
\usepackage{amsmath,amsfonts,bm}
\ifCLASSINFOpdf
\else
   \usepackage[dvips]{graphicx}
\fi
\usepackage{url}
\usepackage{cite}
\usepackage{hyperref}
\usepackage{mathtools}
\usepackage{amssymb}
\usepackage{amsthm}
\usepackage{mathrsfs}
\usepackage{algorithm}
\usepackage{algorithmic}
\usepackage{bbding}
\usepackage{color}
\newtheorem{definition}{\bf Definition}
\newtheorem{assumption}{\bf Assumption}
\newtheorem{theorem}{\bf Theorem}
\newtheorem{lemma}{\bf Lemma}

\newtheorem{proposition}{\bf Proposition}
\newtheorem{corollary}{\bf Corollary}

\usepackage{graphicx}

\begin{document}

\title{Decentralized Optimization on Compact Submanifolds by Quantized Riemannian Gradient Tracking}

\author{Jun Chen$^*$, Lina Liu$^*$, Tianyi Zhu, Yong Liu,~\IEEEmembership{Member,~IEEE}, Guang Dai, Yunliang Jiang, \\ Ivor W.~Tsang,~\IEEEmembership{Fellow,~IEEE}
\thanks{This research was supported by National Natural Science Foundation of China (Grant No. U22A20102), Zhejiang Provincial Natural Science Foundation of China (Grant No. LQN25F030018). \textit{Corresponding authors: Jun Chen and Yunliang Jiang}.}
\thanks{Jun Chen is with the National Special Education Resource Center for Children with Autism, Zhejiang Normal University, Hangzhou 311231, China, and with the School of Computer Science and Technology, Zhejiang Normal University, Jinhua 321004, China (E-mail: junc.change@gmail.com).}
\thanks{Lina Liu and Tianyi Zhu are with China Mobile Research Institute, Beijing 100032, China. (E-mail: liulina0601@gmail.com; zhu-ty@outlook.com).}
\thanks{Yong Liu is with the Institute of Cyber-Systems and Control, Zhejiang University, Hangzhou 310027, China (E-mail: yongliu@iipc.zju.edu.cn).}
\thanks{Guang Dai is with SGIT AI Lab, State Grid Corporation of China, China (E-mail: guang.gdai@gmail.com).}
\thanks{Yunliang Jiang is with the Zhejiang Key Laboratory of Intelligent Education Technology and Application, Zhejiang Normal University, Jinhua 321004, China, and with the School of Computer Science and Technology, Zhejiang Normal University, Jinhua 321004, China, and also with the School of Information Engineering, Huzhou University, Huzhou 313000, China (E-mail: jyl2022@zjnu.cn).}
\thanks{Ivor W.~Tsang is with the Centre for Frontier Artificial Intelligence Research, A*STAR, Singapore (E-mail: ivor.tsang@gmail.com).}
\thanks{$^{*}$Equal contribution.}
}

\markboth{IEEE TRANSACTIONS ON SIGNAL PROCESSING}
{Shell \MakeLowercase{\textit{et al.}}: Bare Demo of IEEEtran.cls for IEEE Journals}
\maketitle

\begin{abstract}
This paper considers the problem of decentralized optimization on compact submanifolds, where a finite sum of smooth (possibly non-convex) local functions is minimized by $n$ agents forming an undirected and connected graph. However, the efficiency of distributed optimization is often hindered by communication bottlenecks. To mitigate this, we propose the Quantized Riemannian Gradient Tracking (Q-RGT) algorithm, where agents update their local variables using quantized gradients. The introduction of quantization noise allows our algorithm to bypass the constraints of the accurate Riemannian projection operator (such as retraction), further improving iterative efficiency. To the best of our knowledge, this is the first algorithm to achieve an $\mathcal{O}(1/K)$ convergence rate in the presence of quantization, matching the convergence rate of methods without quantization. Additionally, we explicitly derive lower bounds on decentralized consensus associated with a function of quantization levels. Numerical experiments demonstrate that Q-RGT performs comparably to non-quantized methods while reducing communication bottlenecks and computational overhead. 
\end{abstract}


\IEEEpeerreviewmaketitle

\section{Introduction}

\IEEEPARstart{I}{n} large-scale systems such as signal processing, control, and machine learning, data is often distributed across numerous nodes, making it challenging for a centralized server to manage the increasing computational demands. As a result, 
distributed optimization has become increasingly important in recent years due to the challenges posed by large-scale datasets and complex multi-agent systems. Among the explored approaches, the distributed sub-gradient method is notable for its simplicity in combining local gradient descent to reduce the consensus error~\cite{nedic2009distributed,yuan2016convergence}. To achieve exact convergence to an $\epsilon$-stationary point, various algorithms have leveraged both local and global historical information. For instance, the gradient tracking algorithm~\cite{qu2017harnessing,yuan2018exact}, primal-dual framework~\cite{alghunaim2020decentralized}, and ADMM~\cite{shi2014linear,aybat2017distributed} have been explored with convex local functions.

\renewcommand\arraystretch{1.2} 
\begin{table}[t]
	\caption{Comparison with existing algorithms.}
	\begin{center}
		\begin{tabular}{ccccc}
			\hline
			  Methods & GT & Operator & Conv. rate  & Compression \\
			\hline
            DRDGD~\cite{chen2021decentralized}  & \XSolidBrush & Retraction & $\mathcal{O}\left(\frac{1}{\sqrt{K}}\right)$ & \XSolidBrush \\
            DPRGD~\cite{deng2023decentralized} & \XSolidBrush & Projection & $\mathcal{O}\left(\frac{1}{\sqrt{K}}\right)$ & \XSolidBrush \\
            DRCGD~\cite{chen2024decentralized} & \XSolidBrush & Projection & / & \XSolidBrush \\
            \hline
            RGT~\cite{chen2021decentralized} & \CheckmarkBold & Retraction & $\mathcal{O}\left(\frac{1}{K}\right)$ & \XSolidBrush \\
            RFIM~\cite{hu2023decentralized} & \CheckmarkBold & Retraction & $\mathcal{O}\left(\frac{1}{K}\right)$ & \XSolidBrush \\
            SRGT~\cite{zhao2024distributed} & \CheckmarkBold & Retraction & $\mathcal{O}\left(\frac{1}{K}\right)$ & \XSolidBrush \\
            PRGT~\cite{deng2023decentralized} & \CheckmarkBold & Projection & $\mathcal{O}\left(\frac{1}{K}\right)$  & \XSolidBrush \\
            PRGC~\cite{hu2024improving} & \CheckmarkBold & Projection & $\mathcal{O}\left(\frac{1}{K}\right)$  & \CheckmarkBold \\
            RFGT~\cite{sun2024global} & \CheckmarkBold & N/A & $\mathcal{O}\left(\frac{1}{K}\right)$  & \XSolidBrush \\
            Q-RGT (\textbf{Ours}) & \CheckmarkBold & N/A  & $\mathcal{O}\left(\frac{1}{K}\right)$  & \CheckmarkBold \\
			\hline
		\end{tabular}
	\end{center}
\label{summary}
\end{table}

Let $\mathcal{M}$ be a compact submanifold of $\mathbb{R}^{d\times r}$. We consider the following distributed non-convex (but smooth) optimization problem on a compact submanifold: 
\begin{equation}
\begin{aligned}
&\min \frac{1}{n} \sum_{i=1}^n f_i\left(x_i\right), \\
&\text { s.t. }  x_1=x_2=\cdots=x_n, \quad x_i \in \mathcal{M}, \quad \forall i \in [n],
\end{aligned}
\label{decentralized}
\end{equation}
where $n$ represents the number of agents, and $f_i$ denotes the local function for each agent. Many important large-scale tasks can be formulated as this distributed optimization problem, such as the principle component analysis~\cite{ye2021deepca}, eigenvalue estimation~\cite{chen2021decentralized}, dictionary learning~\cite{raja2015cloud}, and deep neural networks with orthogonal constraint~\cite{vorontsov2017orthogonality,huang2018orthogonal,eryilmaz2022understanding}.

Despite the recent advancements in distributed optimization on the Stiefel manifold, the issue of communication bottlenecks in this context has received limited attention. Recently, Ref. \cite{hu2024improving} reduced the communication quantities (total number of entries communicated) in all iterations through single-step consensus. This raises another question: Can we design a decentralized Riemannian algorithm to reduce the communication bottleneck and computational overhead through low-precision gradients? In this paper, we propose the Quantized Riemannian Gradient Tracking (Q-RGT) algorithm to efficiently address problem (\ref{decentralized}) over connected networks. A comparison of our work with the existing methods is provided in Table~\ref{summary}. Our contributions can be summarized as follows: 

\begin{enumerate}
    \item Building upon the landing algorithm, we design a novel quantization scheme that rounds Riemannian gradients up or down based on the direction toward the manifold, thereby achieving the manifold constraint.
    \item We propose Q-RGT, a novel quantized distributed algorithm for solving problem (\ref{decentralized}). The algorithm is free from retraction and projection, and iterates within the neighborhood of compact submanifolds.
    \item We establish that Q-RGT achieves a convergence rate of $\mathcal{O}(1/K)$, matching the performance of its retraction-based and non-quantized counterparts. This is the first convergence result in the presence of quantization.
    \item Numerical experiments demonstrate that Q-RGT performs on par with existing methods while significantly reducing communication bottlenecks and computational overhead.
\end{enumerate}

\section{Related Works}

\subsection{Decentralized Optimization}

Initially, Ref. \cite{tsitsiklis1986distributed} introduced an intuitive distributed gradient descent algorithm (DGD) in Euclidean space. To ensure consensus, Ref. \cite{nedic2009distributed} proposed a distributed sub-gradient algorithm for convex problems, which accommodates time-varying and asynchronous communication by requiring a diminishing step-size. Later, Ref. \cite{chen2021distributed} extended this idea by introducing a distributed stochastic sub-gradient descent algorithm for weakly convex problems. To achieve exact convergence with constant step-size, Ref. \cite{shi2015extra} developed a first-order distributed algorithm (EXTRA) for convex Lipschitz-differentiable objectives, whose iteration complexity is comparable to DGD. Ref. \cite{di2016next} utilized a dynamic consensus mechanism to distribute the computation of non-convex problems among agents. Furthermore, Ref. \cite{qu2017harnessing} introduced the gradient tracking method into the framework of DGD, achieving a linear convergence rate for strongly convex objectives. More recently, Ref. \cite{sun2022centralized} developed the distributed mirror descent algorithm, which accelerates convergence in convex settings.

\subsection{Decentralized Optimization on the Stiefel Manifold}

Due to the non-convex nature of the Stiefel manifold, the aforementioned results can not be directly applied in this setting. As a result, existing studies have further developed Riemannian optimization tools. Based on the local linear convergence rate of distributed Riemannian consensus \cite{chen2023local}, Ref. \cite{chen2021decentralized} introduced a retraction-based distributed Riemannian gradient descent method (DRDGD) with a convergence rate of $\mathcal{O}(1/\sqrt{K})$, achieved by using a diminishing step-size. Additionally, Ref. \cite{chen2021decentralized} proposed a Riemannian gradient tracking algorithm (RGT) with a convergence rate of $\mathcal{O}(1/K)$ using a constant step-size. Building on this, Ref. \cite{deng2023decentralized} proposed a projection-based distributed Riemannian gradient method (DPRGD), which has the same convergence rate as DRDGD. A gradient tracking variant (PRGT) was also established, matching the convergence rate of RGT. Subsequently, Ref. \cite{hu2023decentralized} utilized the Riemannian Fisher information matrix (RFIM) to develop a second-order distributed Riemannian natural gradient algorithm, which achieves the best-known convergence rate of $\mathcal{O}(1/K)$ to a stationary point. By combining a variable sample scheme with gradient tracking, Ref. \cite{zhao2024distributed} introduced a retraction-based stochastic gradient tracking method (SRGT), achieving a convergence rate of $\mathcal{O}(1/K)$ in expectation when the sample size increases exponentially. More recently, Ref. \cite{chen2024decentralized} proposed a projection-based distributed Riemannian conjugate gradient method (DRCGD). Since the centralized conjugate gradient algorithm has not yet proven its convergence rate, DRCGD only guarantees global convergence. Additionally, Ref. \cite{wang2022decentralized} achieved a single round of communications by introducing an approximate augmented Lagrangian penalty function. However, as noted in Ref. \cite{ablin2022fast}, a key drawback of penalty methods is that the obtained solution is not generally feasible. To address this limitation, Ref. \cite{sun2024global} proposed a retraction-free gradient tracking algorithm (RFGT) to achieve the global convergence through the landing algorithm \cite{ablin2022fast}.

\subsection{Decentralized Optimization via Quantized Communication}

The limited communication bandwidth among agents is a significant bottleneck in decentralized optimization. To address this, various quantized distributed algorithms have been developed to reduce communication costs. For instance, Ref. \cite{el2016design} proposed a deterministic distributed averaging algorithm with uniform quantization. To achieve vanishing consensus error, Ref. \cite{reisizadeh2019exact} introduced a novel quantized distributed gradient decent algorithm (QDGD) for strongly convex and smooth objectives, in which each node transmits a quantized local decision variable to neighboring nodes under the assumption of unbiasedness. For distributed stochastic learning over a directed graph, Ref. \cite{taheri2020quantized} proposed a quantized algorithm based on the push-sum algorithm for convex and non-convex objectives. Recently, efforts have been made to accelerate the convergence of communication-efficient algorithms. Ref. \cite{kovalev2021linearly} introduced a randomized first-order algorithm with the variance reduction technique, achieving linear convergence for strongly convex distributed problems. Inspired by primal-dual methods, Ref. \cite{liu2021linear} presented a linearly convergent distributed algorithm with communication compression. Moreover, Ref. \cite{xiong2022quantized} developed a quantized distributed gradient tracking algorithm that achieves linear convergence even with one-bit communication. However, communication bottlenecks on compact submanifolds have received limited attention.

\section{Preliminaries}
\subsection{Notations}

\renewcommand\arraystretch{1.2}  
\begin{table}[h]
    \centering
    \caption{Symbol notations}
    \begin{tabular}{c|c}
        \hline
        Symbol & Notation \\
        \hline
        $\mathcal{M}$ & compact submanifold \\
        \hline
        $\mathcal{M}^{\Omega(N)}$ & quantized region \\
        \hline
        $N$ & quantized bit-width \\
        \hline
        $\mathcal{R}_{x}$ & retraction \\
        \hline
        $\mathcal{P}_{T_x \mathcal{M}}$ & orthogonal projection onto $T_x \mathcal{M}$ \\
        \hline
        $T_x \mathcal{M}$, $N_x \mathcal{M}$ & tangent space and normal space of $\mathcal{M}$ at $x$ \\
        \hline
        $\nabla f(x)$, $\operatorname{grad} f(x)$ & Euclidean gradient and Riemannian gradient of $f$ \\
        \hline
        $\operatorname{dist}(x,\mathcal{M})$ & Euclidean distance between $x$ and $\mathcal{M}$ \\
        \hline
        $\Vert \cdot \Vert$ & Frobenius norm \\
        \hline
    \end{tabular}
    \label{sym}
\end{table}

The undirected and connected graph $G=(\mathcal{V},\mathcal{E})$ is defined by the set of agents $\mathcal{V}=\{1,2,\cdots,n\}$ and the set of edges $\mathcal{E}$. The adjacency matrix $W$ of this topological graph satisfies $W_{ij}>0$ and $W_{ij}=W_{ji}$ if there exists an edge $(i,j) \in \mathcal{E}$; otherwise, $W_{ij}=0$. We represent the collection of all local variables $x_i$ as $\mathbf{x}$, stacking them as $\mathbf{x}^\top:=(x_1^\top, \cdots, x_n^\top)$. The function $f(\mathbf{x})$ is defined as the average of local functions: $\frac{1}{n} \sum_{i=1}^n f_i(x_i)$. We use $\emph{I}_d$ to represent the $d \times d$ identity matrix and $\textbf{1}_n \in \mathbb{R}^n$ as a vector with all entries equal to one. Furthermore, we define $\mathbf{W}^t:=W^t \otimes \emph{I}_d$, where $\otimes$ denotes the Kronecker product and $t$ is a positive integer. The symbols are summarized in Table \ref{sym}.

\subsection{Consensus Problem over Networks}

Let $x_i$ represent the local variables of each agent $i \in [n]$. Then, the Euclidean average point of $x_1, x_2, \cdots, x_n$ is denoted by
\begin{equation}
    \bar{x}:=\frac{1}{n} \sum_{i=1}^n x_i .
\end{equation}
In Euclidean space, the consensus error can be measured by $\Vert x_i - \bar{x}\Vert$.
It is important to address the consensus problem, which involves minimizing the following quadratic loss function:
\begin{equation}
\begin{aligned}
&\min \varphi^t(\mathbf{x}):=\frac{1}{4} \sum_{i=1}^n \sum_{j=1}^n W_{i j}^t\left\|x_i-x_j\right\|^2, \\
&\text { s.t. }  x_i, x_j \in \mathcal{M}, \quad \forall i,j \in[n] \text {, } \\
\end{aligned}
\label{consensus}
\end{equation}
where the positive integer $t$ represents the $t$-th power of $W$. Note that the doubly stochastic matrix $W_{ij}^t$ satisfies the following assumption and is calculated by performing $t$ communication steps.

\begin{assumption}
(Connectivity) We assume that for the undirected and connected graph $G$, the doubly stochastic matrix $W$ satisfies: (a)
$W = W^\top$; (b) $0 < W_{ii} < 1$ and $W_{ij} \geq 0$; (c) the eigenvalues of $W$ lie in $(-1, 1]$. In addition, the second largest singular value $\sigma_2$ of $W$ falls within $(0, 1)$.
\label{weight}
\end{assumption}

\subsection{Riemannian Optimization}

In order to develop the Riemannian version of decentralized gradient descent, we introduce a retraction operator based on the following definition to update along a negative Riemannian gradient direction on the tangent space $T_x \mathcal{M}$.

\begin{definition}
(Retraction \cite{absil2008optimization}) A smooth map $\mathcal{R}: T \mathcal{M}  \rightarrow \mathcal{M}$ is called a retraction on a smooth manifold $\mathcal{M}$ if the retraction of $\mathcal{R}$ to the tangent space $T_x \mathcal{M}$ at any point $x \in \mathcal{M}$, denoted by $\mathcal{R}_x$, satisfies the following conditions: (a) $\mathcal{R}$ is continuously differentiable; (b) $\mathcal{R}_x(0_x)=x$, where $0_x$ is the zero element of $T_x \mathcal{M}$;
(c) $\mathrm{D} \mathcal{R}_x(0_x)=\operatorname{id}_{T_x \mathcal{M}}$, the identity mapping.
\label{retraction}
\end{definition}

When the second-order boundedness of retraction is satisfied, it means
\begin{equation}
    \mathcal{R}_x(\eta)=x+\eta +\mathcal{O}(\Vert \eta \Vert^2).
\label{second_order}
\end{equation}
In this case, $\mathcal{R}_x(\eta)$ is a local approximation of $x+\eta$. Consequently, the decentralized Riemannian gradient descent \cite{chen2021decentralized} iterates as
\begin{equation}
    x_{i,k+1}=\mathcal{R}_{x_{i,k}}\left(\mathcal{P}_{T_{x_{i,k}} \mathcal{M}}\left(\sum_{j=1}^n W^t_{ij}x_{j,k}\right) - \alpha \operatorname{grad} f(x_{i,k})\right)
\end{equation}
where $\mathcal{P}_{T_x \mathcal{M}}(\cdot)$ is the orthogonal projection onto $T_x \mathcal{M}$~\cite{edelman1998geometry,absil2008optimization}. On the Stiefel manifold $\operatorname{St}(d,r)=\{x \in \mathbb{R}^{d \times r}|x^\top x = \emph{I}_r\}$ where $r\leq d$~\cite{zhu2017riemannian,sato2022riemannian}, it follows from $y \in \mathbb{R}^{d \times r}$ that
\begin{equation}
\mathcal{P}_{T_x \mathcal{M}}(y) = y - \frac{1}{2} x (x^\top y + y^\top x).
\label{projection}
\end{equation}

The Riemannian gradient on $\mathcal{M}$, utilizing the induced Riemannian metric induced by the Euclidean inner product $\langle \cdot,\cdot \rangle$, can be expressed as
\begin{equation}
\operatorname{grad} f(x)= \mathcal{P}_{T_x \mathcal{M}}(\nabla f(x)).
\end{equation}
We define the distance between $x \in \mathbb{R}^{d\times r}$ and $\mathcal{M}$ by
\begin{equation}
\operatorname{dist}(x,\mathcal{M}):=\inf_{y \in \mathcal{M}} \Vert y-x\Vert.
\end{equation}
Since any compact $C^2$-submanifolds in Euclidean space belong to proximally smooth set \cite{clarke1995proximal}, we define an $R$-proximally smooth set $\mathcal{M}$ that satisfies
\begin{equation}
    \langle v, y-x\rangle \leq \frac{\|v\|}{2R}\|y-x\|^2 , \quad \forall x, y \in \mathcal{M}, v \in N_x \mathcal{M} .
\label{inner}
\end{equation}
The Stiefel manifold $\operatorname{St}(d,r)$ exhibits $1$-proximal smoothness, whereas the Grassmann manifold $\operatorname{Gr}(d,r)$ is $1/\sqrt{2}$-proximal smoothness \cite{balashov2021gradient}.

\subsection{Lipschitz Assumption}

We assume that the global objective function is Lipschitz smooth \cite{zhang2016first}, as well as the local objective function.
\begin{assumption}
\label{lipschitz}
(Lipschitz Smoothness) We assume that $f(x)$ is differentiable and $L$-Lipschitz smooth on $\mathbb{R}^{d\times r}$, i.e., for all $x,y \in \mathbb{R}^{d\times r}$, we have
\begin{equation}
    f_i(y) \leq f_i(x) + \langle \nabla f_i(x), y-x \rangle + \frac{L}{2} \Vert y-x \Vert^2 , \quad i \in [n].
\end{equation}
\end{assumption}
With Assumption \ref{lipschitz} in Euclidean space and the properties of proximally smooth sets, we can derive Lipschitz smooth on the manifold. Given Proposition \ref{pro1} on the manifold, the Riemannian gradient has the upper bound, i.e., $\Vert\operatorname{grad}f_i(x)\Vert \leq L_g$.

\begin{proposition}
Under Assumption~\ref{lipschitz}, for any $x,y \in \mathcal{M}$, if $f(x)$ is $L$-Lipschitz smooth in Euclidean space, there exists a constant $L_g=L+L_f$ such that
\begin{equation}
    f_i(y) \leq f_i(x) + \langle \operatorname{grad} f_i(x), y-x \rangle + \frac{L_g}{2} \Vert y-x \Vert^2 ,
\end{equation}
where $L_f:= \frac{1}{R} \max_{x \in \mathcal{M}}\Vert \nabla f_i(x) \Vert$.
\label{pro1}
\end{proposition}
\begin{proof}
Since $\operatorname{grad} f_i(x)=\mathcal{P}_{T_x \mathcal{M}}(\nabla f_i(x))$, we have
\begin{equation}
    \langle \operatorname{grad} f_i(x), y-x \rangle  = \langle \nabla f_i(x), y-x \rangle - \langle \mathcal{P}_{N_x \mathcal{M}} (\nabla f_i(x)), y-x \rangle
\end{equation}
where the projection satisfies $\mathcal{P}_{T_x \mathcal{M}} (\nabla f_i(x)) = \nabla f_i(x) - \mathcal{P}_{N_x \mathcal{M}} (\nabla f_i(x))$.
It follows from Assumption~\ref{lipschitz} that
\begin{equation}
\begin{aligned}
    & f_i (y) - f_i(x) - \langle \operatorname{grad} f_i(x), y-x \rangle \\
    & = f_i (y) - f_i(x) - \langle \nabla f_i(x), y-x \rangle \\
    & + \langle \mathcal{P}_{N_x \mathcal{M}} (\nabla f_i(x)), y-x \rangle \\
    & \leq \frac{L}{2} \Vert y-x\Vert^2 + \frac{\Vert\mathcal{P}_{N_x \mathcal{M}} (\nabla f_i(x))\Vert}{2R} \Vert y-x \Vert^2 \\
    & \leq \left(\frac{L}{2}+\frac{1}{2R} \max_{x \in \mathcal{M}}\Vert \nabla f_i(x) \Vert\right) \Vert y-x \Vert^2 \\
    & \leq \frac{L + L_f}{2} \Vert y-x \Vert^2 ,
\end{aligned}
\end{equation}
where the first inequality utilizes Eq.(\ref{inner}). The proof is completed.
\end{proof}

\section{Quantized Riemannian Gradient Tracking}

\begin{figure}[ht]
\centering
\includegraphics[width=1\linewidth]{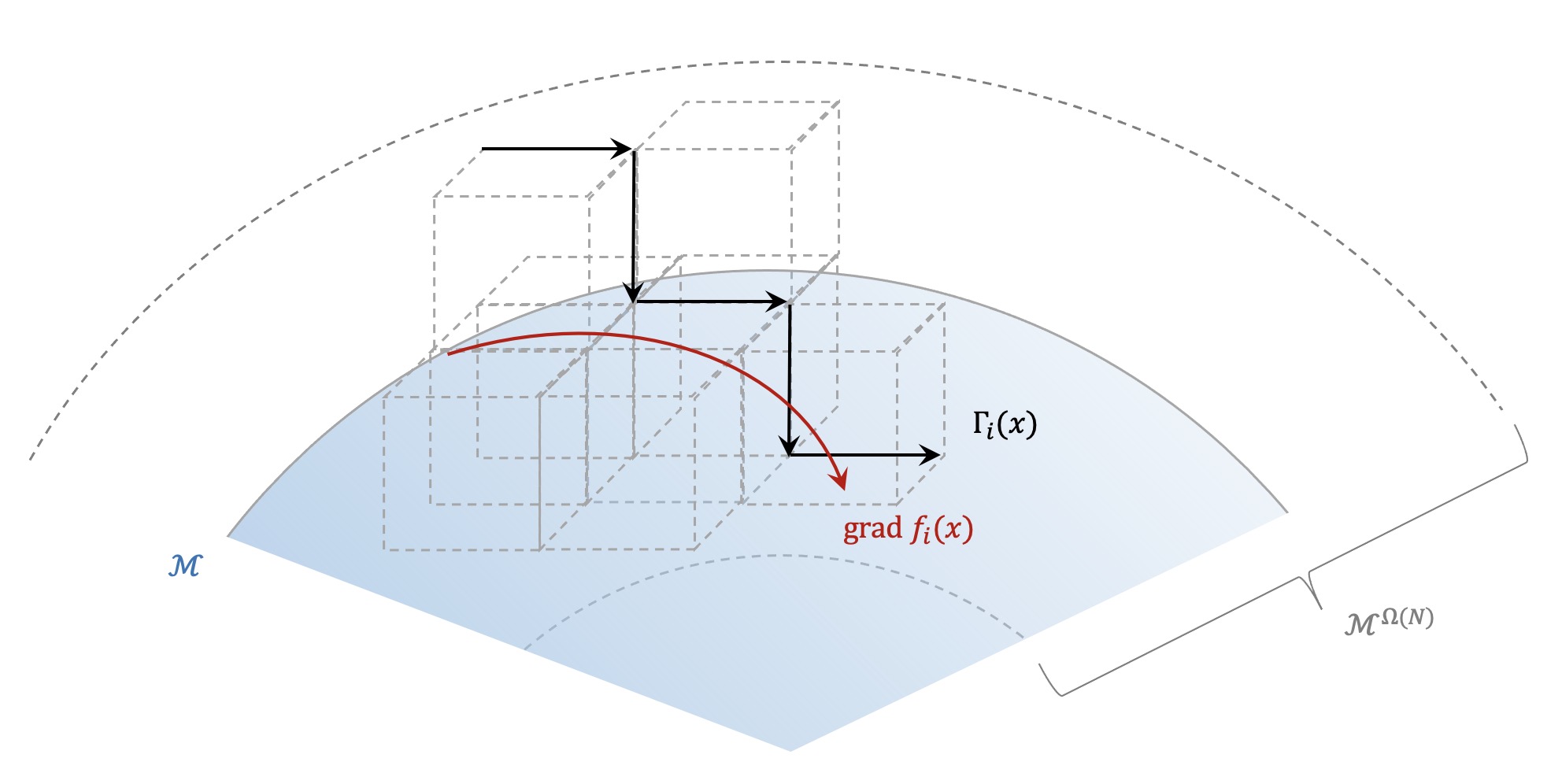}
\caption{Illustration of the geometry of the quantized gradient $\Gamma$ in the quantized region $\mathcal{M}^{\Omega(N)}$. Note that $\mathcal{M}^{\Omega(N)}$ is in a neighborhood of the compact submanifold $\mathcal{M}$.}
\label{fig2}
\end{figure}

\subsection{The Algorithm}

In this paper, we explore the compression of the Riemannian gradient using a uniform quantizer. A widely used quantization method is rounding-to-nearest, where a value is rounded to the nearest quantization grid point using the round function $\lfloor \cdot \rceil$, thereby minimizing the quantization error. Additionally, a value can be rounded down using the floor function $\lfloor \cdot \rfloor$ or rounded up using the ceiling function $\lceil \cdot \rceil$. To achieve uniform quantization over an $N$-bit fixed-point grid, we define the rounding-to-nearest quantizer as follows:
\begin{equation}
\begin{aligned}
    \tilde{\mathcal{Q}}_N(\cdot) &= \gamma(\cdot) \left[\operatorname{round}_N\left(\frac{\cdot}{\gamma(\cdot)} + \frac{1}{2}\right) - \frac{1}{2} \right] \\
    & \operatorname{where} \quad \operatorname{round}_N(x) := \frac{\lfloor x \cdot (2^N-1)\rceil}{2^N-1}.
\label{round1}
\end{aligned}
\end{equation}
Here, the scale factor $\gamma(\cdot)=2 \max (|\cdot|)$ is introduced to normalize values to the range $[-0.5,0.5]$. By shifting this range to $[0,1]$ through the addition of $0.5$, the round function can be applied directly within a uniform quantizer. After quantization, values are restored to their original range via inverse scaling. Notably, although the scale factor is used during quantization, it does not need to be additionally stored and communicated.

In fact, applying the retraction to the quantized Riemannian gradient may introduce some pathological issues. For example, the quantized version of the zero element may no longer be exactly zero, violating condition (b) in Definition \ref{retraction}. Additionally, the quantization noise may affect the second-order boundedness described in Eq. (\ref{second_order}).

Considering the landing algorithm \cite{ablin2022fast,sun2024global}, a retraction-free algorithm iterates as
\begin{equation}
    x_{k+1} = x_{k} - \alpha \Lambda(x_{k}).
\end{equation}
On a compact submanifold $\mathcal{M}$, the landing field $\Lambda$ is computed as
\begin{equation}
\begin{aligned}
    \Lambda(x) &:=\operatorname{grad} f(x)+\lambda \nabla \mathcal{N}(x) \\
    & \operatorname{where} \quad \mathcal{N}(x)=\operatorname{Dist}(x,\mathcal{M})^2.
\label{landing}
\end{aligned}
\end{equation}
In this case, the landing algorithm attracts $x$ towards $\mathcal{M}$ based on the term $\lambda \nabla \mathcal{N}(x)$. In particular, we have $\mathcal{N}(x)=\Vert x^\top x - \mathit{I}_r \Vert^2$ on the Stiefel manifold $\operatorname{St}(d,r)$.

Inspired by the landing algorithm, this work proposes a novel uniform quantizer to replace the role of the landing field closing to the manifold. Instead of using the predominant rounding-to-nearest approach for quantizing the Riemannian gradient, we determine whether to round up or down based on $\nabla \mathcal{N}(x)$. Specifically, we have
\begin{equation}
\begin{aligned}
    & \hat{\mathcal{Q}}_N(\cdot) = \gamma(\cdot) \left[\operatorname{round}_N\left(\frac{\cdot}{\gamma(\cdot)} + \frac{1}{2}\right) - \frac{1}{2} \right] \\
    & \operatorname{where} \;\; \operatorname{round}_N(x) := \frac{\lfloor x \cdot (2^N-1)\rfloor+\lfloor \operatorname{sigmoid}{\nabla \mathcal{N}(x)}\rceil}{2^N-1}.
\end{aligned}
\end{equation}
Here, $\operatorname{sigmoid}(x)=\frac{1}{1+e^{-x}}$ is employed to normalize values to the range $(0,1)$, which are then rounded to either $0$ or $1$ using $\lfloor \cdot \rceil$. Compared to Eq.(\ref{round1}), our method rounds down when $\nabla \mathcal{N}(x)$ is negative and up when $\nabla \mathcal{N}(x)$ is positive. As a result, $\hat{\mathcal{Q}}_N(\operatorname{grad} f(x))$ not only achieves quantization but also effectively serves the role of the landing field.

Furthermore, we introduce a uniform distribution $\mathcal{U}_N \sim \mathcal{U}(-\frac{0.5}{2^N-1}, +\frac{0.5}{2^N-1})$ to reduce the potential bias caused by quantization, ensuring the noise range is equivalent to the quantization error. Consequently, we express the uniform quantizer:
\begin{equation}
\begin{aligned}
    &\mathcal{Q}_N(\cdot) = \gamma(\cdot) \left[\operatorname{round}_N\left(\frac{\cdot}{\gamma(\cdot)} + \frac{1}{2} +\mathcal{U}_N\right) - \frac{1}{2} \right] \\ 
    & \operatorname{where} \;\; \operatorname{round}_N(x) := \frac{\lfloor x \cdot (2^N-1)\rfloor+\lfloor \operatorname{sigmoid}{\nabla \mathcal{N}(x)}\rceil}{2^N-1}.
\end{aligned}
\end{equation}
We require that the following assumption is satisfied.
\begin{assumption}
\label{unbias}
(Unbiasedness) The uniform quantizer $\mathcal{Q}(\cdot)$ is unbiased and has a bounded variance $\nu^2$:
\begin{equation}
\mathbb{E}[\mathcal{Q}(\mathbf{x})] = \mathbf{x}, \quad \operatorname{and} \quad \mathbb{E}[\Vert\mathcal{Q}(\mathbf{x})-\mathbf{x}\Vert^2] = \nu^2 .
\end{equation}
For any $\mathbf{x}$, the quantization is carried out independently on distributed nodes.
\end{assumption}

Given the uniform $N$-bit quantizer, the quantized Riemannian gradient can be expressed as the following form:
\begin{equation}
    \Gamma (x)=\mathcal{Q}_N(\operatorname{grad} f(x)),
\label{QRG}
\end{equation}
where the uniform quantizer $\mathcal{Q}_N(\cdot)$ for a gradient vector $[\operatorname{grad} f_1(x),\cdots,\operatorname{grad} f_n(x)]^\top$ is defined as $[\mathcal{Q}_N(\operatorname{grad} f_1(x)),\cdots,\mathcal{Q}_N(\operatorname{grad} f_n(x))]^\top$.

Compared to Eq. (\ref{landing}), the quantized Riemannian gradient in Eq. (\ref{QRG}) can be rewritten as a similar form
\begin{equation}
    \Gamma (x) = \operatorname{grad} f(x) + \frac{C}{2^N-1}\nabla \mathcal{N}(x) + \mathcal{U}_N ,
\label{similar}
\end{equation}
where $C=\frac{\gamma(\operatorname{grad} f(x))}{\gamma(\nabla \mathcal{N}(x))}$. The quantized gradient can be regarded as a stochastic landing field that stays randomly within a neighborhood of the manifold, ensuring that each iterate remains close to the manifold. We refer to this neighborhood as the quantized region, characterized by a safety factor $\Omega(N)$, as illustrated in Figure \ref{fig2}.

\begin{definition}
    (Quantized Region) A neighborhood of the compact submanifold, for $\Omega(N)$, satisfies:
    \begin{equation}
        \mathcal{M}^{\Omega(N)}:=\{x \in \mathbb{R}^{d \times r}| \operatorname{Dist}(x,\mathcal{M}) \leq \Omega(N) \},
    \end{equation}
    where $N$ represents the bit-width and the safety factor satisfies $\Omega(N)=\frac{C}{2^N-1}$ with a constant $C$.
\end{definition}

We show that the iteration $x_{k+1}=x_k -\alpha \Gamma(x_k)$ converges to the quantized region if and only if $x_0 \in \mathcal{M}^{\Omega(N)}$.
\begin{proposition}
    For the iteration $x_{k+1}=x_k -\alpha \Gamma(x_k)$ where $\Gamma(x)$ is defined in Eq.(\ref{similar}), the algorithm converges to the quantized region $\mathcal{M}^{\Omega(N)}$.
\end{proposition}
\begin{proof}
    Using the function $\mathcal{N}(x)=\operatorname{Dist}(x,\mathcal{M})^2$, it follows that $\nabla \mathcal{N}(x)=0$ when $x \in \mathcal{M}$. This condition serves as the sole criterion for the algorithm's termination. A simple approach to design an algorithm that converges to $\mathcal{M}^{\Omega(N)}$ consists in following $-\nabla \mathcal{N}(x)$, which leads in the discrete setting to Potter's algorithm $x_{k+1}=x_k - \frac{C\alpha}{2^N-1} \nabla \mathcal{N}(x)$ \cite{cardoso1996equivariant}. This algorithm can be regarded as a generalization of the case on the Stiefel manifold \cite{ablin2022fast}. The proof is completed.
\end{proof}

\begin{algorithm}[htbp]
	\caption{Quantized Riemannian Gradient Tracking}
	\label{alg2}
	\begin{algorithmic}[1]
		\REQUIRE
		Initial point $\mathbf{x}_0 \in \mathcal{M}^{\Omega(N)}$, an integer $N\in [1,32]$, an integer $t$, the step size $\alpha$. \\
            \algorithmiccomment{for each node $i \in [n]$, in parallel}
		\FOR{$k=0,\cdots$}
            \STATE Update $x_{i,k+1}=\sum_{j=1}^n W_{i j}^t x_{j,k}-\alpha s_{i,k}$
            \STATE Compute Riemannian gradient $\operatorname{grad} f(x_{i,k+1})$
            \STATE Quantize the gradient $\Gamma_{i,k+1}=\mathcal{Q}_N(\operatorname{grad} f(x_{i,k+1}))$
            \STATE Quantized Riemannian gradient tracking \\ $s_{i,k+1}=\sum_{j=1}^n W_{i j}^t s_{j,k} + \Gamma_{i,k+1} - \Gamma_{i,k}$
		\ENDFOR
	\end{algorithmic}
\end{algorithm}

By the Lipschitz smoothness in Proposition \ref{pro1}, we define the following relationship on the quantized region
\begin{equation}
    L_m=\max(L_g, \max_{x\in \mathcal{M}^{\Omega(N)}} \Vert \nabla f(x) \Vert) .
\label{lm}
\end{equation}
Additionally, let $L_{\Gamma}$ denote the Lipschitz smoothness factor of the quantized Riemannian gradient $\Gamma(x)$. Then, the smoothness factors satisfy the following relationship: $L_g \leq L_{\Gamma} \leq L_m$.

Consequently, with the help of the gradient tracking algorithm, we propose the Quantized Riemannian gradient tracking method (Q-RGT):
\begin{equation}
    s_{i,k+1}=\sum_{j=1}^n W_{i j}^t s_{j,k} + \Gamma_{i,k+1} - \Gamma_{i,k}.
\end{equation}
And the exact update iterates as
\begin{equation}
    x_{i,k+1}=\sum_{j=1}^n W_{i j}^t x_{j,k}-\alpha s_{i,k}.
\end{equation}
The Q-RGT method is presented in Algorithm~\ref{alg2}, which is efficient in both computation and communication.

\subsection{Technical Lemmas}

For ease of notation, let us denote
\[
\begin{aligned}
&\bar{x}_k:=\frac{1}{n} \sum_{i=1}^n x_{i,k}, \quad \bar{s}_k:=\frac{1}{n} \sum_{i=1}^n s_{i,k}, \quad \bar{\Gamma}_k:=\frac{1}{n} \sum_{i=1}^n \Gamma_{i,k} \\
&\bar{\mathbf{x}}_k:=\mathbf{1} \otimes \bar{x}_k, \quad \bar{\mathbf{s}}_k:=\mathbf{1} \otimes \bar{s}_k, \quad \bar{\mathbf{\Gamma}}_k:=\mathbf{1} \otimes \bar{\Gamma}_k .
\end{aligned}
\]
We present that the iterate of Euclidean average $\bar{x}$ remains in the quantized region $\mathcal{M}^{\Omega(N)}$ and the consensus error satisfies $\Vert \mathbf{x} - \bar{\mathbf{x}}\Vert \leq \Omega(N)$ as long as the step size is small enough.
\begin{proposition}
    \cite{sun2024global,ablin2024infeasible} Let $\bar{x} \in \mathcal{M}^{\Omega(N)}$ and $\Vert \mathbf{x} - \bar{\mathbf{x}}\Vert \leq \Omega(N)$. Assuming that $\operatorname{grad} f(x) \leq L_g$, if the step size $\alpha$ is small enough, the next iterate satisfies $\bar{x}^* \in \mathcal{M}^{\Omega(N)}$ and $\Vert \mathbf{x}^* - \bar{\mathbf{x}}^*\Vert \leq \Omega(N)$.
\end{proposition}

Firstly, we consider the stability of Algorithm \ref{alg2} in terms of a linear system. We write this algorithm in a matrix formulation:
\begin{equation}
\begin{aligned}
    & \mathbf{x}_{k+1}=\mathbf{W} \mathbf{x}_{k} - \alpha \mathbf{s}_{k} \\
    & \mathbf{s}_{k+1} = \mathbf{W} \mathbf{s}_{k} + \mathbf{\Gamma}(\mathbf{x}_{k+1}) - \mathbf{\Gamma}(\mathbf{x}_{k})
\end{aligned}
\label{matrix}
\end{equation}
Next that we present the following lemmas to bound the consensus errors $\Vert \mathbf{x}_k - \bar{\mathbf{x}}_k \Vert^2$ and $\Vert \mathbf{s}_k - \bar{\mathbf{s}}_k \Vert^2$.

\begin{lemma}
\label{x}
    Under Assumption~\ref{weight}. Let $\{\mathbf{x}_k\}$ be generated by Eq.(\ref{matrix}), then the following inequality on $\mathbf{x}$ holds:
    \[
    \begin{aligned}
       & \Vert \mathbf{x}_k - \bar{\mathbf{x}}_k \Vert^2 \\
       & \leq \frac{(1+\sigma_2)\sigma_2}{2} \Vert\mathbf{x}_{k-1} - \bar{\mathbf{x}}_{k-1}\Vert^2 + \frac{1+\sigma_2}{1-\sigma_2}\alpha^2 \Vert \mathbf{s}_{k-1} - \bar{\mathbf{s}}_{k-1} \Vert^2 .
    \end{aligned}
    \]
\end{lemma}
\begin{proof}
By the definition of $\mathbf{x}_k$, we have
\begin{equation}
\begin{aligned}
    & \Vert \mathbf{x}_k - \bar{\mathbf{x}}_k \Vert^2=\Vert \mathbf{x}_k - \mathbf{1}\otimes\bar{x}_k \Vert^2 \\
    &=\Vert \mathbf{W} (\mathbf{x}_{k-1} - \bar{\mathbf{x}}_{k-1})- \alpha (\mathbf{s}_{k-1} - \bar{\mathbf{s}}_{k-1}) \Vert^2 \\
    & \leq \frac{1+\sigma_2}{2\sigma_2} \Vert \mathbf{W} (\mathbf{x}_{k-1} - \bar{\mathbf{x}}_{k-1})\Vert^2 + \frac{1+\sigma_2}{1-\sigma_2}\alpha^2 \Vert \mathbf{s}_{k-1} - \bar{\mathbf{s}}_{k-1} \Vert^2 \\
    & \leq \frac{(1+\sigma_2)\sigma_2}{2} \Vert\mathbf{x}_{k-1} - \bar{\mathbf{x}}_{k-1}\Vert^2 + \frac{1+\sigma_2}{1-\sigma_2}\alpha^2 \Vert \mathbf{s}_{k-1} - \bar{\mathbf{s}}_{k-1} \Vert^2,
\end{aligned}
\end{equation}
where Cauchy–Schwarz inequality and the definition of $\mathbf{W}$ are applied. The proof is completed.
\end{proof}

\begin{lemma}
\label{s}
    Under Assumption~\ref{weight}. Let $\{\mathbf{x}_k\}$ be generated by Eq.(\ref{matrix}), then the following inequality on $\mathbf{s}$ holds:
    \[
    \begin{aligned}
        & \Vert \mathbf{s}_k - \bar{\mathbf{s}}_k \Vert^2 \\
        & \leq \left(\frac{(1+\sigma_2)\sigma_2}{2} + 4 L_m^2 \frac{1+\sigma_2}{1-\sigma_2}\alpha^2\right) \Vert \mathbf{s}_{k-1} - \bar{\mathbf{s}}_{k-1}\Vert^2 \\
        & + 2 L_m^2 \frac{1+\sigma_2}{1-\sigma_2}\sigma_2^2 \Vert \mathbf{x}_{k-1} - \bar{\mathbf{x}}_{k-1}\Vert^2 + 4 L_m^2 \frac{1+\sigma_2}{1-\sigma_2} \alpha^2 \Vert \bar{\mathbf{s}}_{k-1} \Vert^2 .
    \end{aligned}
    \]
\end{lemma}
\begin{proof}
It follows from $L_g \leq L_{\Gamma} \leq L_m$ that
\begin{equation}
\begin{aligned}
    & \Vert \mathbf{s}_k - \bar{\mathbf{s}}_k \Vert^2 \\
    & = \Vert \mathbf{W} (\mathbf{s}_{k-1} - \bar{\mathbf{s}}_{k-1}) + (\mathbf{\Gamma}_k - \bar{\mathbf{\Gamma}}_k - \mathbf{\Gamma}_{k-1} + \bar{\mathbf{\Gamma}}_{k-1})\Vert^2 \\
    & \leq \frac{(1+\sigma_2)\sigma_2}{2}\Vert \mathbf{s}_{k-1} - \bar{\mathbf{s}}_{k-1}\Vert^2 \\
    & + \frac{1+\sigma_2}{1-\sigma_2} \Vert \mathbf{\Gamma}_k - \bar{\mathbf{\Gamma}}_k - \mathbf{\Gamma}_{k-1} + \bar{\mathbf{\Gamma}}_{k-1}\Vert^2 \\
    & \leq  \frac{(1+\sigma_2)\sigma_2}{2} \Vert \mathbf{s}_{k-1} - \bar{\mathbf{s}}_{k-1}\Vert^2 + \frac{1+\sigma_2}{1-\sigma_2}\Vert \mathbf{\Gamma}_k - \mathbf{\Gamma}_{k-1}\Vert^2 \\
    & \leq \frac{(1+\sigma_2)\sigma_2}{2} \Vert \mathbf{s}_{k-1} - \bar{\mathbf{s}}_{k-1}\Vert^2 + \frac{1+\sigma_2}{1-\sigma_2} L_m^2 \Vert \mathbf{x}_k - \mathbf{x}_{k-1}\Vert^2 \\
    & \leq \frac{(1+\sigma_2)\sigma_2}{2} \Vert \mathbf{s}_{k-1} - \bar{\mathbf{s}}_{k-1}\Vert^2 \\
    & + \frac{1+\sigma_2}{1-\sigma_2} L_m^2 \Vert (\mathbf{W}-\mathbf{I})\mathbf{x}_{k-1} - \alpha \mathbf{s}_{k-1}\Vert^2 \\
    & \leq \frac{(1+\sigma_2)\sigma_2}{2} \Vert \mathbf{s}_{k-1} - \bar{\mathbf{s}}_{k-1}\Vert^2 + 2L_m^2 \frac{1+\sigma_2}{1-\sigma_2} \alpha^2 \Vert \mathbf{s}_{k-1} \Vert^2\\
    & + 2 L_m^2 \frac{1+\sigma_2}{1-\sigma_2} \sigma_2^2 \Vert \mathbf{x}_{k-1} - \bar{\mathbf{x}}_{k-1}\Vert^2 ,
\end{aligned}    
\end{equation}
where we use the same technique as Lemma \ref{x}. Next, it follows from the decomposition of $\Vert \mathbf{s}_{k-1} \Vert^2$ that
\begin{equation}
\begin{aligned}
    & \Vert \mathbf{s}_{k-1} \Vert^2 = \Vert \mathbf{s}_{k-1} - \bar{\mathbf{s}}_{k-1} + \bar{\mathbf{s}}_{k-1} \Vert^2  \\
    & \leq 2 \Vert \mathbf{s}_{k-1} - \bar{\mathbf{s}}_{k-1} \Vert^2 + 2 \Vert \bar{\mathbf{s}}_{k-1} \Vert^2 .
\end{aligned}
\end{equation}
Combining the above two inequalities, we have
\begin{equation}
\begin{aligned}
    & \Vert \mathbf{s}_k - \bar{\mathbf{s}}_k \Vert^2 \\
    & \leq \left(\frac{(1+\sigma_2)\sigma_2}{2} + 4 L_m^2 \frac{1+\sigma_2}{1-\sigma_2}\alpha^2\right) \Vert \mathbf{s}_{k-1} - \bar{\mathbf{s}}_{k-1}\Vert^2 \\
    & + 2 L_m^2 \frac{1+\sigma_2}{1-\sigma_2}\sigma_2^2 \Vert \mathbf{x}_{k-1} - \bar{\mathbf{x}}_{k-1}\Vert^2 + 4 L_m^2 \frac{1+\sigma_2}{1-\sigma_2} \alpha^2 \Vert \bar{\mathbf{s}}_{k-1} \Vert^2 .
\end{aligned}
\end{equation}
The proof is completed.
\end{proof}

With the above two inequalities on the consensus errors, we prove the stability of Algorithm \ref{alg2} in the following theorem.

\begin{theorem}
\label{stable}
    When Lemma~\ref{x} and \ref{s} are satisfied, we can write $\Vert \mathbf{x}_k - \bar{\mathbf{x}}_k\Vert^2$ and $\Vert \mathbf{s}_k - \bar{\mathbf{s}}_k\Vert^2$ as a linear system with input signal $\Upsilon$, state $\Phi$ and transition $\Lambda$,
    \begin{equation}
        \Phi_k \leq \Lambda \Phi_{k-1} + \Upsilon_{k-1},
    \end{equation}
    where
    \begin{equation}
    \begin{aligned}
        \Phi_k &:= \begin{bmatrix} \frac{\Vert \mathbf{s}_k - \bar{\mathbf{s}}_k\Vert^2}{L_m} \\ L_m \Vert \mathbf{x}_k - \bar{\mathbf{x}}_k\Vert^2
        \end{bmatrix}, \\
        \Lambda &:= \begin{bmatrix}  \frac{(1+\sigma_2)\sigma_2}{2} + 4 L_m^2 \frac{1+\sigma_2}{1-\sigma_2}\alpha^2 & 2 \frac{1+\sigma_2}{1-\sigma_2}\sigma_2^2 \\ \frac{1+\sigma_2}{1-\sigma_2}\alpha^2 L_m^2 & \frac{(1+\sigma_2)\sigma_2}{2}
        \end{bmatrix}, \\
        \Upsilon_{k-1} &:= \begin{bmatrix} 4 L_m \frac{1+\sigma_2}{1-\sigma_2} \alpha^2 \Vert \bar{\mathbf{s}}_{k} \Vert^2 \\ 0
        \end{bmatrix} .
    \end{aligned}
    \end{equation}
    If $\alpha < \frac{(1-\sigma_2)^2}{4 L_m}$, then the linear system is stable.
\end{theorem}
\begin{proof}
    In order to ensure that the linear system is stable, i.e. $\rho(\Lambda) < 1$, if and only if there exists
    \begin{equation}
    \begin{aligned}
        \left(\frac{(1+\sigma_2)\sigma_2}{2} + 4 L_m^2 \frac{1+\sigma_2}{1-\sigma_2}\alpha^2\right) y_1 + 2 \frac{1+\sigma_2}{1-\sigma_2}\sigma_2^2 y_2 & \leq y_1 \\
        \frac{1+\sigma_2}{1-\sigma_2}\alpha^2 L_m^2 y_1 + \frac{(1+\sigma_2)\sigma_2}{2} y_2 & \leq y_2 .
    \end{aligned}
    \end{equation}
    Therefore, we have
    \begin{equation}
    \begin{aligned}
        & 2 \frac{1+\sigma_2}{1-\sigma_2}\sigma_2^2 y_2 \leq\left(\frac{(1-\sigma_2)(\sigma_2+2)}{2} - 4 L_m^2 \frac{1+\sigma_2}{1-\sigma_2}\alpha^2\right) y_1 \\
        & 2 \frac{(1+\sigma_2)^2}{(1-\sigma_2)(\sigma_2+2)}\alpha^2 L_m^2 y_1 \leq y_2 .
    \end{aligned}
    \end{equation}
    Let $y_2 = 8 \frac{1+\sigma_2}{1-\sigma_2} \alpha^2 L_m^2 y_1$. With some inequality relaxation, we obtain the safety step size $\alpha$ such that
    \begin{equation}
        \alpha \leq \frac{(1-\sigma_2)^2}{4 L_m}.
    \end{equation}
    The proof is completed.
\end{proof}

Theorem \ref{stable} means that under the safety step size, the linear system on $\mathbf{x}$ and $\mathbf{s}$ is stable. Based on this, we can bound the accumulated consensus error on $\Vert \mathbf{x}_k - \bar{\mathbf{x}}_k\Vert^2$.

\begin{corollary}
\label{x}
    Under a safety step size satisfying $\alpha \leq \frac{(1-\sigma_2)^2}{16 L_m}$, the following boundedness holds:
    \[
    \begin{aligned}
        & \sum_{k=0}^K \Vert \mathbf{x}_k - \bar{\mathbf{x}}_k\Vert^2 \leq 32\frac{1+\sigma_2}{(1-\sigma_2)^3} \alpha^4 L_m^2 \sum_{k=0}^K \Vert \bar{\mathbf{s}}_{k} \Vert^2  \\
        & + 8\frac{1+\sigma_2}{(1-\sigma_2)^3} \alpha^2 \Vert \mathbf{s}_0 - \bar{\mathbf{s}}_0\Vert^2 .
    \end{aligned}
    \]
\end{corollary}
\begin{proof}
    We first compute
    \begin{equation}
        \mathit{I}-\Lambda = \begin{bmatrix} 1-\frac{(1+\sigma_2)\sigma_2}{2} - 4 \frac{1+\sigma_2}{1-\sigma_2}\alpha^2 L_m^2 & -2 \frac{1+\sigma_2}{1-\sigma_2}\sigma_2^2 \\ -\frac{1+\sigma_2}{1-\sigma_2}\alpha^2 L_m^2 & 1-\frac{(1+\sigma_2)\sigma_2}{2}
        \end{bmatrix} .
    \end{equation}
    Then, the determinant of $\mathit{I}-\Lambda$ yields
    \begin{equation}
    \begin{aligned}
        & \det(\mathit{I}-\Lambda) \\
        & = \left(1-\frac{(1+\sigma_2)\sigma_2}{2} - 4 L_m^2 \frac{1+\sigma_2}{1-\sigma_2}\alpha^2\right) \frac{(1-\sigma_2)(\sigma_2+2)}{2} \\
        & - 2 \frac{(1+\sigma_2)^2}{(1-\sigma_2)^2} \sigma_2^2 \alpha^2 L_m^2 .
    \end{aligned}
    \end{equation}
    Since $\alpha \leq \frac{(1-\sigma_2)^2}{16 L_m}$, we can give the lower bound: $\det(\mathit{I}-\Lambda) \geq \frac{(1-\sigma_2)^2}{8}$. Therefore, we have
    \begin{equation}
    \begin{aligned}
        & (\mathit{I}-\Lambda)^{-1} = \frac{(\mathit{I}-\Lambda)^*}{\det(\mathit{I}-\Lambda)} \\
        & \leq \frac{8}{(1-\sigma_2)^2} \begin{bmatrix} 1-\frac{(1+\sigma_2)\sigma_2}{2} & 2 \frac{1+\sigma_2}{1-\sigma_2}\sigma_2^2 \\ \frac{1+\sigma_2}{1-\sigma_2}\alpha^2 L_m^2 & 1-\frac{(1+\sigma_2)\sigma_2}{2} - 4 \frac{1+\sigma_2}{1-\sigma_2}\alpha^2 L_m^2
        \end{bmatrix} \\
        & = \begin{bmatrix} 4\frac{\sigma_2+2}{1-\sigma_2} & 16\frac{1+\sigma_2}{(1-\sigma_2)^3} \sigma_2^2 \\ 8\frac{1+\sigma_2}{(1-\sigma_2)^3}\alpha^2 L_m^2 & 4\frac{\sigma_2+2}{1-\sigma_2} - 32\frac{1+\sigma_2}{(1-\sigma_2)^3} \alpha^2 L_m^2
        \end{bmatrix} .
    \end{aligned}
    \end{equation}
    On the other hand, it follows from the linear system $\Phi_k \leq \Lambda \Phi_{k-1} + \Upsilon_{k-1}$ that $\Phi_k \leq \Lambda^k \Phi_{0} + \sum_{t=0}^k \Lambda^t \Upsilon_{k-t-1}$. Therefore, we have
    \begin{equation}
    \begin{aligned}
        \sum_{k=0}^K \Phi_k & \leq \sum_{k=0}^K \Lambda^k \Phi_{0} + \sum_{k=0}^K\sum_{t=0}^k \Lambda^t \Upsilon_{k-t-1} \\
        & \leq \sum_{k=0}^{\infty} \Lambda^k \Phi_{0} + \sum_{k=0}^K\sum_{t=0}^{\infty} \Lambda^t \Upsilon_{k-t-1} \\
        & = (\mathit{I}-\Lambda)^{-1} \left(\Phi_{0}+\sum_{k=0}^K \Upsilon_{k}\right) .
    \end{aligned}
    \end{equation}
    Furthermore, it follows from $\Vert \mathbf{x}_0 - \bar{\mathbf{x}}_0\Vert^2=0$ that
    \begin{equation}
    \begin{aligned}
        & \sum_{k=0}^K L_m \Vert \mathbf{x}_k - \bar{\mathbf{x}}_k\Vert^2 \leq  8\frac{1+\sigma_2}{(1-\sigma_2)^3}\alpha^2 L_m^2 \frac{\Vert \mathbf{s}_0 - \bar{\mathbf{s}}_0\Vert^2}{L_m} \\
        & +  8\frac{1+\sigma_2}{(1-\sigma_2)^3}\alpha^2 L_m^2 \sum_{k=0}^K 4 L_m \frac{1+\sigma_2}{1-\sigma_2}\alpha^2 \Vert \bar{\mathbf{s}}_{k} \Vert^2 .
    \end{aligned}
    \end{equation}
    Consequently, we have
    \begin{equation}
    \begin{aligned}
        & \sum_{k=0}^K \Vert \mathbf{x}_k - \bar{\mathbf{x}}_k\Vert^2 \leq 32\frac{1+\sigma_2}{(1-\sigma_2)^3} \alpha^4 L_m^2 \sum_{k=0}^K \Vert \bar{\mathbf{s}}_{k} \Vert^2  \\
        & + 8\frac{1+\sigma_2}{(1-\sigma_2)^3} \alpha^2 \Vert \mathbf{s}_0 - \bar{\mathbf{s}}_0\Vert^2 .
    \end{aligned}
    \end{equation}
    The proof is completed.
\end{proof}

\subsection{Convergence Analysis}

By utilizing the Lipschitz-type inequalities on the quantized region, we can show the following descent lemma on $f$.

\begin{lemma}
\label{iteration}
    Under Assumption \ref{weight}, \ref{lipschitz} and \ref{unbias}, suppose $\alpha \leq \frac{2}{L_m}$. It follows that
    \begin{equation}
    \begin{aligned}
        & \mathbb{E}[f(\bar{x}_k)] \leq f(\bar{x}_{k-1}) - \frac{\alpha}{2}  \Vert\operatorname{grad} f(\bar{x}_{k-1}) \Vert^2 \\
        & + \frac{\alpha^2 L_m}{2} \mathbb{E}[\Vert \bar{s}_{k-1} \Vert^2] + \frac{L_m}{n} \Vert \mathbf{x}_{k-1} - \bar{\mathbf{x}}_{k-1} \Vert^2 .
    \end{aligned}
    \end{equation}
\end{lemma}

\begin{proof}
By invoking Proposition~\ref{pro1} and Eq.(\ref{lm}), we have
\begin{equation}
\begin{aligned}
    & \mathbb{E}[f(\bar{x}_k)] - f(\bar{x}_{k-1}) \\
    & \leq \langle \operatorname{grad} f(\bar{x}_{k-1}), \mathbb{E}[\bar{x}_{k}-\bar{x}_{k-1}] \rangle + \frac{L_m}{2} \mathbb{E}[\Vert \bar{x}_{k}-\bar{x}_{k-1} \Vert^2] \\
    & = - \alpha \langle \operatorname{grad} f(\bar{x}_{k-1}), \mathbb{E}[\bar{s}_{k-1}] \rangle + \frac{\alpha^2 L_m}{2} \mathbb{E}[\Vert \bar{s}_{k-1} \Vert^2] \\
    & = - \alpha \langle \operatorname{grad} f(\bar{x}_{k-1}), \mathbb{E}[\Gamma(\bar{x}_{k-1})] \rangle + \frac{\alpha^2 L_m}{2} \mathbb{E}[\Vert \bar{s}_{k-1} \Vert^2] \\
    & + \alpha \langle \operatorname{grad} f(\bar{x}_{k-1}), \mathbb{E}[\Gamma(\bar{x}_{k-1}) - \bar{s}_{k-1}] \rangle \\
    & = - \alpha  \Vert\operatorname{grad} f(\bar{x}_{k-1}) \Vert^2 + \frac{\alpha^2 L_m}{2} \mathbb{E}[\Vert \bar{s}_{k-1} \Vert^2] \\
    & + \alpha \langle \operatorname{grad} f(\bar{x}_{k-1}), \mathbb{E}[\Gamma(\bar{x}_{k-1})-\bar{s}_{k-1}] \rangle \\
    & \leq - \alpha  \Vert\operatorname{grad} f(\bar{x}_{k-1}) \Vert^2 + \frac{\alpha^2 L_m}{2} \mathbb{E}[\Vert \bar{s}_{k-1} \Vert^2] \\
    & + \frac{\alpha^2 L_m}{4} \Vert \operatorname{grad} f(\bar{x}_{k-1}) \Vert^2 + \frac{1}{L_m} \Vert\mathbb{E}[\bar{s}_{k-1} - \Gamma(\bar{x}_{k-1})] \Vert^2 \\
    & \leq - \frac{\alpha}{2}  \Vert\operatorname{grad} f(\bar{x}_{k-1}) \Vert^2 + \frac{\alpha^2 L_m}{2} \mathbb{E}[\Vert \bar{s}_{k-1} \Vert^2] \\
    & + \frac{1}{L_m} \Vert\mathbb{E}[\bar{s}_{k-1} - \Gamma(\bar{x}_{k-1})] \Vert^2 ,
\end{aligned}
\end{equation}
where the third equality uses $\mathbb{E}[\Gamma(\bar{x}_{k-1})]=\operatorname{grad} f(\bar{x}_{k-1})$ based on Assumption~\ref{unbias}, the second inequality utilizes Young's inequality, and the third inequality uses $\alpha < \frac{2}{L_m}$. It follows from Jensen's inequality that
\begin{equation}
\begin{aligned}
    & \Vert\mathbb{E}[\bar{s}_{k-1} - \Gamma(\bar{x}_{k-1})] \Vert^2 \leq \mathbb{E}[\Vert\bar{s}_{k-1} - \Gamma(\bar{x}_{k-1}) \Vert^2] \\
    & \leq \frac{L^2_m}{n} \Vert \mathbf{x}_{k-1} - \mathbf{1}\otimes \bar{x}_{k-1} \Vert^2 = \frac{L^2_m}{n} \Vert \mathbf{x}_{k-1} - \bar{\mathbf{x}}_{k-1} \Vert^2 .
\end{aligned}
\end{equation}
The proof is completed.
\end{proof}

With the Lemma \ref{iteration}, we can present the following two theorems, which establish an $\mathcal{O}(1/K)$ convergence rate in Algorithm \ref{alg2}.

\begin{theorem}
\label{converge}
    Suppose that Theorem~\ref{stable} and Lemma~\ref{iteration} hold. If step size $\alpha < \min \left\{\frac{1}{8L_m},\frac{(1-\sigma_2)^2}{16 L_m},\frac{1}{16L_m}\sqrt{\frac{n(1-\sigma_2)^3}{2L_m^2+1}}\right\}$ is satisfied, we get the convergence rate:
    \begin{equation}
        \min_{k=0,\cdots,K} \Vert\mathbb{E}[\Gamma(\bar{x}_{k})] \Vert^2 = \mathcal{O} \left(\frac{1}{\alpha K}\right)
    \end{equation}
\end{theorem}

\begin{proof}
According to $\Vert \bar{s}_{k-1} \Vert^2=\Vert \bar{s}_{k-1} - \Gamma(\bar{x}_{k-1}) + \Gamma(\bar{x}_{k-1}) \Vert^2$, we have
\begin{equation}
\begin{aligned}
    & \mathbb{E}[\Vert \bar{s}_{k-1} \Vert^2]=\mathbb{E}[ \Vert \bar{s}_{k-1} - \Gamma(\bar{x}_{k-1}) + \Gamma(\bar{x}_{k-1}) \Vert^2] \\
    & \leq 2\Vert \operatorname{grad} f(\bar{x}_{k-1}) \Vert^2 + 2\mathbb{E}[\Vert \bar{s}_{k-1} - \Gamma(\bar{x}_{k-1}) \Vert^2] \\
    & \leq 2\Vert \operatorname{grad} f(\bar{x}_{k-1}) \Vert^2 + \frac{2L^2_m}{n} \Vert \mathbf{x}_{k-1} - \bar{\mathbf{x}}_{k-1} \Vert^2 .
\end{aligned}
\end{equation}
It follows from Lemma~\ref{iteration} that
\begin{equation}
\begin{aligned}
    & \mathbb{E}[f(\bar{x}_k)] \\
    & \leq f(\bar{x}_{k-1}) - \frac{\alpha}{2}  \Vert\operatorname{grad} f(\bar{x}_{k-1}) \Vert^2 - \frac{\alpha^2 L_m}{2} \mathbb{E}[\Vert \bar{s}_{k-1} \Vert^2] \\
    & + 2\alpha^2 L_m \Vert \operatorname{grad} f(\bar{x}_{k-1}) \Vert^2 + \frac{2\alpha^2 L^3_m}{n} \Vert \mathbf{x}_{k-1} - \bar{\mathbf{x}}_{k-1} \Vert^2\\
    & + \frac{L_m}{n} \Vert \mathbf{x}_{k-1} - \bar{\mathbf{x}}_{k-1} \Vert^2 \\
    & \leq f(\bar{x}_{k-1}) - \frac{\alpha}{4}  \Vert\operatorname{grad} f(\bar{x}_{k-1}) \Vert^2 - \frac{\alpha^2 L_m}{2} \mathbb{E}[\Vert \bar{s}_{k-1} \Vert^2] \\
    & + \frac{2\alpha^2 L^3_m + L_m}{n} \Vert \mathbf{x}_{k-1} - \bar{\mathbf{x}}_{k-1} \Vert^2 ,
\end{aligned}
\end{equation}
where the second inequality uses $\alpha < \frac{1}{8L_m}$. By rearranging the above inequality, we have
\begin{equation}
\begin{aligned}
    & \frac{\alpha}{4}  \Vert\mathbb{E}[\Gamma(\bar{x}_{k-1})] \Vert^2 \leq f(\bar{x}_{k-1}) - \mathbb{E}[f(\bar{x}_k)] \\
    & - \frac{\alpha^2 L_m}{2} \mathbb{E}[\Vert \bar{s}_{k-1} \Vert^2] + \frac{2\alpha^2 L^3_m + L_m}{n} \Vert \mathbf{x}_{k-1} - \bar{\mathbf{x}}_{k-1} \Vert^2 .
\end{aligned}
\end{equation}
Summing both sides and taking the expectation give us
\begin{equation}
\begin{aligned}
    & \sum_{k=0}^K \frac{\alpha}{4}  \Vert\mathbb{E}[\Gamma(\bar{x}_{k})] \Vert^2 \\
    & \leq f(\bar{x}_{0}) - f^*  - \frac{\alpha^2 L_m}{2} \sum_{k=0}^K \Vert \bar{s}_{k} \Vert^2 \\
    & + \frac{2\alpha^2 L^3_m + L_m}{n} \sum_{k=0}^K\Vert \mathbf{x}_{k} - \bar{\mathbf{x}}_{k} \Vert^2 \\
    & \leq f(\bar{x}_{0}) - f^*  - \frac{\alpha^2 L_m}{2} \sum_{k=0}^K \Vert \bar{s}_{k} \Vert^2 \\
    & + \frac{2\alpha^2 L^3_m + L_m}{n}\frac{1+\sigma_2}{(1-\sigma_2)^3} 32 \alpha^4 L_m^2 \sum_{k=0}^K \Vert \bar{\mathbf{s}}_{k} \Vert^2 \\
    & + \frac{2\alpha^2 L^3_m + L_m}{n}\frac{1+\sigma_2}{(1-\sigma_2)^3} 8 \alpha^2 \Vert \mathbf{s}_0 - \bar{\mathbf{s}}_0\Vert^2 \\
    & \leq f(\bar{x}_{0}) - f^*  + \frac{16\alpha^2(2\alpha^2 L^3_m + L_m)}{n(1-\sigma_2)^3} \Vert \mathbf{s}_0 - \bar{\mathbf{s}}_0\Vert^2 \\
    & - \left(\frac{\alpha^2 L_m}{2} - \frac{64 \alpha^4 L_m^2(2\alpha^2 L^3_m + L_m)}{n(1-\sigma_2)^3} \right)\sum_{k=0}^K \Vert \bar{s}_{k} \Vert^2 ,
\end{aligned}
\end{equation}
where $f^*=\min_{x \in \operatorname{St}(d,r)} f(x)$ and the second inequality introduces Corollary~\ref{x}. Since $\alpha \leq \frac{1}{16L_m}\sqrt{\frac{n(1-\sigma_2)^3}{2L_m^2+1}}$ such that $\frac{\alpha^2 L_m}{2} - \frac{64 \alpha^4 L_m^2(2\alpha^2 L^3_m + L_m)}{n(1-\sigma_2)^3} \geq 0$, we have
\begin{equation}
    \sum_{k=0}^K \frac{\alpha}{4}  \Vert\mathbb{E}[\Gamma(\bar{x}_{k})] \Vert^2 \leq f(\bar{x}_{0}) - f^*  + \frac{16\alpha^2(2\alpha^2 L^3_m + L_m)}{n(1-\sigma_2)^3} \Vert \mathbf{s}_0 - \bar{\mathbf{s}}_0\Vert^2 .
\end{equation}
This implies
\begin{equation}
    \min_{k=0,\cdots,K} \Vert\mathbb{E}[\Gamma(\bar{x}_{k})] \Vert^2 \leq \frac{4(f(\bar{x}_{0}) - f^* + C_1)}{\alpha K},
\end{equation}
where $C_1=\frac{16\alpha^2(2\alpha^2 L^3_m + L_m)}{n(1-\sigma_2)^3}\Vert \mathbf{s}_0 - \bar{\mathbf{s}}_0\Vert^2$. The proof is completed.
\end{proof}

\begin{corollary}
    Suppose that the conditions for Theorem~\ref{converge} is satisfied. If $\alpha \leq \sqrt[4]{n(1-\sigma_2)^3}\frac{1}{16L_m}$, then the consensus is bounded with
    \begin{equation}
        \min_{k=0,\cdots,K} \mathbb{E}[\Vert \mathbf{x}_k - \bar{\mathbf{x}}_k\Vert^2]= \mathcal{O} \left(\frac{1}{K}\right).
    \end{equation}
\end{corollary}
\begin{proof}
It follows from Corollary~\ref{x} that
\begin{equation}
    \begin{aligned}
        & \sum_{k=0}^K \Vert \mathbf{x}_k - \bar{\mathbf{x}}_k\Vert^2 \leq 32\frac{1+\sigma_2 }{(1-\sigma_2)^3} \alpha^4 L_m^2 \sum_{k=0}^K \Vert \bar{s}_{k} \Vert^2  \\
        & + 8\frac{1+\sigma_2}{(1-\sigma_2)^3} \alpha^2 \Vert \mathbf{s}_0 - \bar{\mathbf{s}}_0\Vert^2 \\
        & \leq 8\frac{1+\sigma_2}{(1-\sigma_2)^3} \alpha^2 \Vert \mathbf{s}_0 - \bar{\mathbf{s}}_0\Vert^2 \\
        & + 64\frac{1+\sigma_2 }{(1-\sigma_2)^3} \alpha^4 L_m^2 \sum_{k=0}^K \Vert \Gamma(\bar{x}_{k}) \Vert^2 \\
        & + 64\frac{1+\sigma_2 }{(1-\sigma_2)^3} \alpha^4 L_m^2 \sum_{k=0}^K \Vert \bar{s}_{k} - \Gamma(\bar{x}_{k}) \Vert^2 \\
        & \leq 8\frac{1+\sigma_2}{(1-\sigma_2)^3} \alpha^2 \Vert \mathbf{s}_0 - \bar{\mathbf{s}}_0\Vert^2 \\
        & + 64\frac{1+\sigma_2 }{(1-\sigma_2)^3} \alpha^4 L_m^2 \sum_{k=0}^K \Vert \Gamma(\bar{x}_{k}) \Vert^2 \\
        & + 64\frac{1+\sigma_2 }{n(1-\sigma_2)^3} \alpha^4 L_m^4 \sum_{k=0}^K \Vert \mathbf{x}_k - \bar{\mathbf{x}}_k \Vert^2 .
    \end{aligned}
\end{equation}
Since we have $\alpha \leq \sqrt[4]{n(1-\sigma_2)^3}\frac{1}{16L_m}$, it follows from $64\frac{1+\sigma_2 }{n(1-\sigma_2)^3} \alpha^4 L_m^4 \leq \frac{1}{2}$ that
\begin{equation}
    \begin{aligned}
        & \sum_{k=0}^K \mathbb{E}[\Vert \mathbf{x}_k - \bar{\mathbf{x}}_k\Vert^2] \leq 16\frac{1+\sigma_2}{(1-\sigma_2)^3} \alpha^2 \Vert \mathbf{s}_0 - \bar{\mathbf{s}}_0\Vert^2 \\
        & + 128\frac{1+\sigma_2 }{(1-\sigma_2)^3} \alpha^4 L_m^2 \sum_{k=0}^K \Vert \mathbb{E}[\Gamma(\bar{x}_{k})] \Vert^2 \\
        & \leq 4C_2(f(\bar{x}_{0}) - f^* + C_1)+ C_3 ,
    \end{aligned}
\end{equation}
where $C_2=128\frac{1+\sigma_2 }{(1-\sigma_2)^3} \alpha^3 L_m^2$ and $C_3=16\frac{1+\sigma_2}{(1-\sigma_2)^3} \alpha^2 \Vert \mathbf{s}_0 - \bar{\mathbf{s}}_0\Vert^2$. This implies
\begin{equation}
    \min_{k=0,\cdots,K} \mathbb{E}[\Vert \mathbf{x}_k - \bar{\mathbf{x}}_k\Vert^2]= \mathcal{O} \left(\frac{1}{K}\right).
\end{equation}
The proof is completed.
\end{proof}

\section{Experiments}

\begin{figure*}[ht]
	\centering
	\begin{minipage}{0.44\linewidth}
		\centering
		\includegraphics[width=1\linewidth]{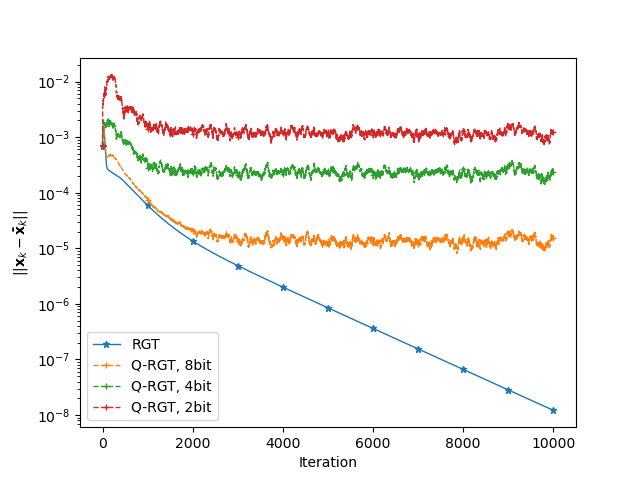}
	\end{minipage}
	\centering
	\begin{minipage}{0.44\linewidth}
		\centering
		\includegraphics[width=1\linewidth]{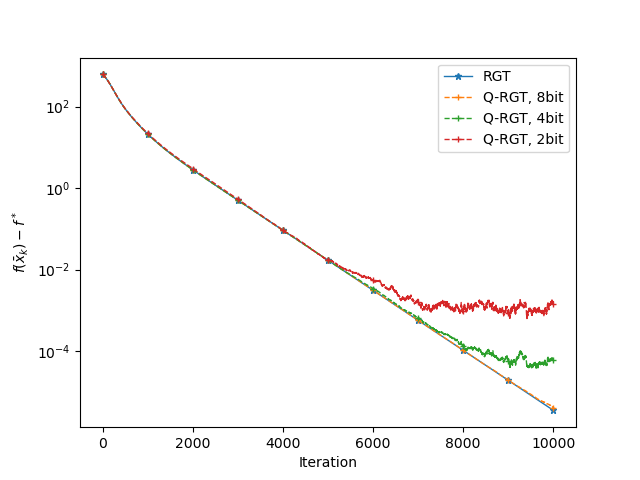}
	\end{minipage}
        \centering
	\begin{minipage}{0.44\linewidth}
		\centering
		\includegraphics[width=1\linewidth]{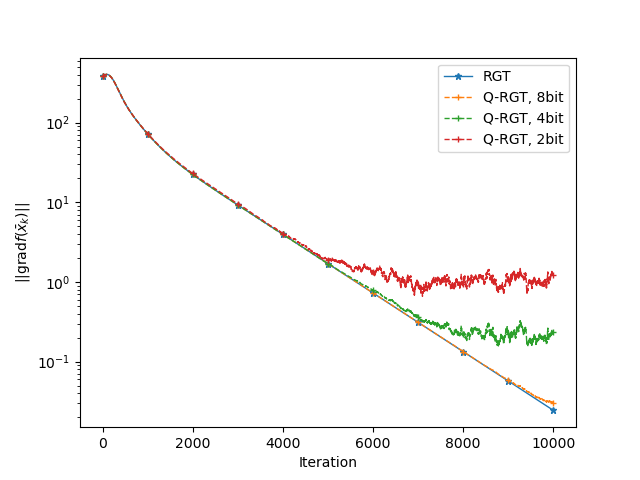}
	\end{minipage}
	\centering
	\begin{minipage}{0.44\linewidth}
		\centering
		\includegraphics[width=1\linewidth]{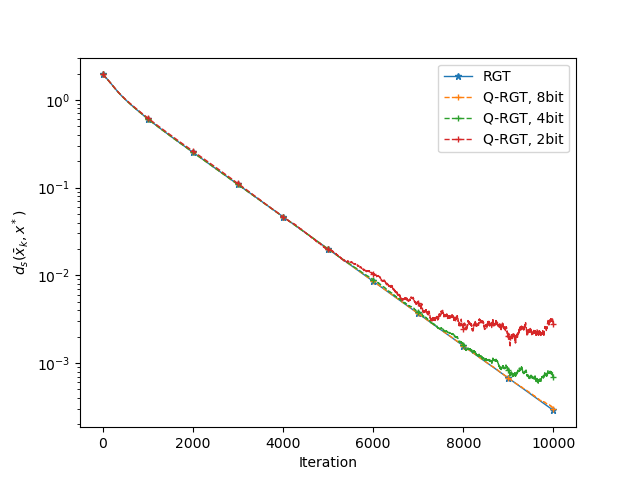}
	\end{minipage}
	\caption{Numerical results on synthetic data with different bit-widths and single-step consensus, eigengap $\Delta = 0.8$, Graph: Ring, $n=16$, $\hat{\alpha}=0.01$.}
	\label{fig1}
\end{figure*}

\begin{figure*}[ht]
	\centering
	\begin{minipage}{0.44\linewidth}
		\centering
		\includegraphics[width=1\linewidth]{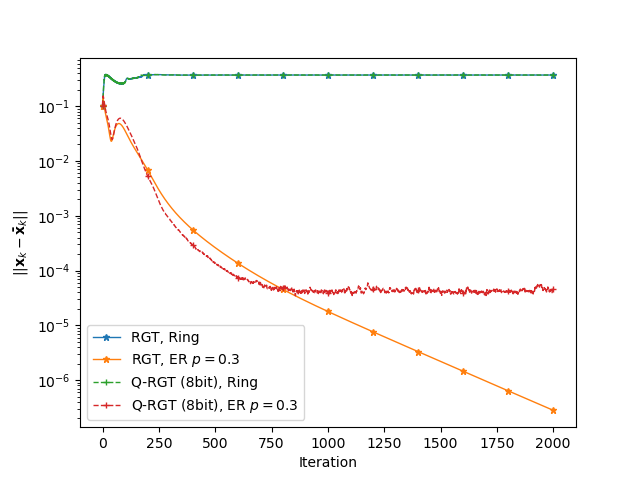}
	\end{minipage}
	\centering
	\begin{minipage}{0.44\linewidth}
		\centering
		\includegraphics[width=1\linewidth]{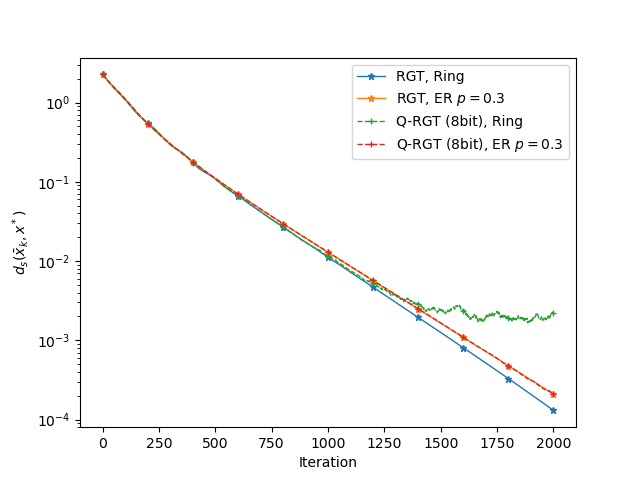}
	\end{minipage}
	\caption{Numerical results on MNIST data with different network graphs and single-step consensus, eigengap $\Delta = 0.8$, $n=16$, $\hat{\alpha}=0.01$.}
	\label{fig2}
\end{figure*}

In this section, we compare our Q-RGT method with RGT~\cite{chen2021decentralized}. We evaluate their performance on the distributed eigenvector problem given by:
\begin{equation}
\begin{aligned}
& \min _{\mathbf{x} \in \operatorname{St}^n}-\frac{1}{2 n} \sum_{i=1}^n \operatorname{tr}\left(x_i^{\top} A_i^{\top} A_i x_i\right), \quad \text { s.t. } \; x_1=x_2=\ldots=x_n ,\\
& \text{where} \quad \operatorname{St}^n:=\underbrace{\operatorname{St}(d, r) \times \operatorname{St}(d, r) \times \cdots \times \operatorname{St}(d, r)}_n .
\label{pca}
\end{aligned}
\end{equation}
Here, $A_i \in \mathbb{R}^{m_i \times d}$ represents the local data matrix, and $m_i$ denotes the sample size for agent $i$. Building upon this, $A^\top:=[A_1^\top, \cdots, A_n^\top]$ denotes the global data matrix. For any solution $x^*$ of Eq.(\ref{pca}), which provides  an orthogonal matrix $O \in \mathbb{R}^{r \times r}$, $x^* O$ also qualifies as a solution in essence. Consequently, the distance between two points $x$ and $x^*$ can be defined as
\begin{equation}
d_s(x,x^*)=\min_{O^\top O=OO^\top=\emph{I}_r} \Vert x O - x^*\Vert.
\end{equation}

We measure these algorithms using four metrics: the consensus error ``$ \Vert \mathbf{x}_k - \mathbf{\bar{x}}_k \Vert$", the gradient norm ``$\Vert \mathrm{grad} f( {\bar{x}}_k) \Vert$", the difference of the objective function ``$f( {\bar{x}}_k) - f^*$", and the distance to the global optimum ``$d_s({\bar{x}}_k, x^*)$". According to \cite{chen2021decentralized}, $W$ is the Metroplis constant matrix~\cite{shi2015extra}. The experiments are conducted using the Intel(R) Core(TM) i7-14700K CPU, and the code is implemented in Python under the package of mpi4py.

\subsection{Synthetic data}

We run $n=16$ and set $m_1=m_2=\cdots=m_{16}=1000$, $d=10$, and $r=5$. Subsequently, we generate $m_1 \times n$ i.i.d samples to obtain $A$ by following standard multi-variate Gaussian distribution. Specifically, let $A= \mathbb{U} \Sigma \mathbb{V}^\top$ represent the truncated SVD, where $\mathbb{V} \in \mathbb{R}^{d \times d}$ and $\mathbb{U} \in \mathbb{R}^{1000n \times d}$ are orthogonal matrices, and $\Sigma \in \mathbb{R}^{d \times d}$ is a diagonal matrix. Finally, we define the singular values of $A$ as $\Sigma_{i,i}=\Sigma_{0,0} \times \Delta^{i/2}$, where $i \in [d]$ and the eigengap satisfies $\Delta \in (0,1)$. We use a constant step size for all comparisons. The step size is defined as $\alpha=\frac{n \hat{\alpha}}{\sum_{i=1}^n m_i}$. We evaluate the graph matrix representing the topology among agents (Ring network). For RGT and Q-RGT, we cap the maximum epoch at 10,000, terminating early if $d_s(\bar{x}_k,x^*) \leq 10^{-8}$.

We conduct ablation experiments on Q-RGT using different bit-widths (2bit, 4bit, 8bit) in a Ring network of 16 agents. The single-step consensus $t=1$ is set, and the step size is $\hat{\alpha}=0.01$. As shown in Figure \ref{fig1}, larger bit-widths consistently results in better performance, which aligns with intuition. This indicates that bit-width affects the lower bound of convergence, which is consistent with our definition of the quantized region. Notably, with an 8bit quantizer, Q-RGT achieves performance comparable to RGT, except for the metric ``$ \Vert \mathbf{x}_k - \mathbf{\bar{x}}_k \Vert$". Even with a 2bit quantizer, our algorithm can still work. Importantly, before reaching the lower bound of convergence, Q-RGT maintains the same convergence rate as RGT, which is consistent with our theoretical analysis.

\subsection{Real-world data}

We present numerical results on the MNIST dataset~\cite{lecun1998mnist}. The MNIST dataset consists of 60,000 hand-written images, each represented as a vector of dimension $d=784$. These samples form a data matrix of size $60000 \times 784$, which is randomly and evenly partitioned among $n$ agents. The data matrix is normalized by dividing 255. Consequently, each agent holds a local data matrix $A_i$ of $\frac{60000}{n} \times 784$. For brevity, we fix $d=784$ and $r=5$. The communication topology among agents is modeled both the Ring network and the Erdos-Renyi (ER) network with probability $p=0.3$. The step size is defined as $\alpha_k=\frac{\hat{\alpha}}{60000}$. For RGT and Q-RGT, we cap the maximum epoch at 2,000, terminating early if $d_s(\bar{x}_k,x^*) \leq 10^{-8}$.

Figure~\ref{fig2} represents the results for MNIST data with $n = 16$. We observe that Q-RGT maintains the same convergence rate as RGT before reaching the lower bound of convergence.

\section{Conclusion}

This paper proposes the Quantized Riemanian Gradient Tracking (Q-RGT) algorithm for solving distributed optimization problems under orthogonality constraints. Notably, this is the first distributed Riemannian optimization algorithm to incorporate quantization. We demonstrate that Q-RGT achieves a convergence rate of $\mathcal{O}(1/K)$, matching that of RGT without quantization. Numerical experiments on eigenvalue problems are presented to validate our throretical analysis.

\bibliographystyle{IEEEtran}
\bibliography{sample}

\begin{thebibliography}{10}
\providecommand{\url}[1]{#1}
\csname url@samestyle\endcsname
\providecommand{\newblock}{\relax}
\providecommand{\bibinfo}[2]{#2}
\providecommand{\BIBentrySTDinterwordspacing}{\spaceskip=0pt\relax}
\providecommand{\BIBentryALTinterwordstretchfactor}{4}
\providecommand{\BIBentryALTinterwordspacing}{\spaceskip=\fontdimen2\font plus
\BIBentryALTinterwordstretchfactor\fontdimen3\font minus \fontdimen4\font\relax}
\providecommand{\BIBforeignlanguage}[2]{{%
\expandafter\ifx\csname l@#1\endcsname\relax
\typeout{** WARNING: IEEEtran.bst: No hyphenation pattern has been}%
\typeout{** loaded for the language `#1'. Using the pattern for}%
\typeout{** the default language instead.}%
\else
\language=\csname l@#1\endcsname
\fi
#2}}
\providecommand{\BIBdecl}{\relax}
\BIBdecl

\bibitem{nedic2009distributed}
A.~Nedic and A.~Ozdaglar, ``Distributed subgradient methods for multi-agent optimization,'' \emph{IEEE Transactions on Automatic Control}, vol.~54, no.~1, pp. 48--61, 2009.

\bibitem{yuan2016convergence}
K.~Yuan, Q.~Ling, and W.~Yin, ``On the convergence of decentralized gradient descent,'' \emph{SIAM Journal on Optimization}, vol.~26, no.~3, pp. 1835--1854, 2016.

\bibitem{qu2017harnessing}
G.~Qu and N.~Li, ``Harnessing smoothness to accelerate distributed optimization,'' \emph{IEEE Transactions on Control of Network Systems}, vol.~5, no.~3, pp. 1245--1260, 2017.

\bibitem{yuan2018exact}
K.~Yuan, B.~Ying, X.~Zhao, and A.~H. Sayed, ``Exact diffusion for distributed optimization and learning—part i: Algorithm development,'' \emph{IEEE Transactions on Signal Processing}, vol.~67, no.~3, pp. 708--723, 2018.

\bibitem{alghunaim2020decentralized}
S.~A. Alghunaim, E.~K. Ryu, K.~Yuan, and A.~H. Sayed, ``Decentralized proximal gradient algorithms with linear convergence rates,'' \emph{IEEE Transactions on Automatic Control}, vol.~66, no.~6, pp. 2787--2794, 2020.

\bibitem{shi2014linear}
W.~Shi, Q.~Ling, K.~Yuan, G.~Wu, and W.~Yin, ``On the linear convergence of the admm in decentralized consensus optimization,'' \emph{IEEE Transactions on Signal Processing}, vol.~62, no.~7, pp. 1750--1761, 2014.

\bibitem{aybat2017distributed}
N.~S. Aybat, Z.~Wang, T.~Lin, and S.~Ma, ``Distributed linearized alternating direction method of multipliers for composite convex consensus optimization,'' \emph{IEEE Transactions on Automatic Control}, vol.~63, no.~1, pp. 5--20, 2017.

\bibitem{chen2021decentralized}
S.~Chen, A.~Garcia, M.~Hong, and S.~Shahrampour, ``Decentralized riemannian gradient descent on the stiefel manifold,'' in \emph{International Conference on Machine Learning}.\hskip 1em plus 0.5em minus 0.4em\relax PMLR, 2021, pp. 1594--1605.

\bibitem{deng2023decentralized}
K.~Deng and J.~Hu, ``Decentralized projected riemannian gradient method for smooth optimization on compact submanifolds,'' \emph{arXiv preprint arXiv:2304.08241}, 2023.

\bibitem{chen2024decentralized}
\BIBentryALTinterwordspacing
J.~Chen, H.~Ye, M.~Wang, T.~Huang, G.~Dai, I.~Tsang, and Y.~Liu, ``Decentralized riemannian conjugate gradient method on the stiefel manifold,'' in \emph{The Twelfth International Conference on Learning Representations}, 2024. [Online]. Available: \url{https://openreview.net/forum?id=PQbFUMKLFp}
\BIBentrySTDinterwordspacing

\bibitem{hu2023decentralized}
J.~Hu, K.~Deng, N.~Li, and Q.~Li, ``Decentralized riemannian natural gradient methods with kronecker-product approximations,'' \emph{arXiv preprint arXiv:2303.09611}, 2023.

\bibitem{zhao2024distributed}
J.~Zhao, X.~Wang, and J.~Lei, ``Distributed riemannian stochastic gradient tracking algorithm on the stiefel manifold,'' \emph{arXiv preprint arXiv:2405.16900}, 2024.

\bibitem{hu2024improving}
J.~Hu and K.~Deng, ``Improving the communication in decentralized manifold optimization through single-step consensus and compression,'' \emph{arXiv preprint arXiv:2407.08904}, 2024.

\bibitem{sun2024global}
Y.~Sun, S.~Chen, A.~Garcia, and S.~Shahrampour, ``Global convergence of decentralized retraction-free optimization on the stiefel manifold,'' \emph{arXiv preprint arXiv:2405.11590}, 2024.

\bibitem{ye2021deepca}
H.~Ye and T.~Zhang, ``Deepca: Decentralized exact pca with linear convergence rate,'' \emph{The Journal of Machine Learning Research}, vol.~22, no.~1, pp. 10\,777--10\,803, 2021.

\bibitem{raja2015cloud}
H.~Raja and W.~U. Bajwa, ``Cloud k-svd: A collaborative dictionary learning algorithm for big, distributed data,'' \emph{IEEE Transactions on Signal Processing}, vol.~64, no.~1, pp. 173--188, 2015.

\bibitem{vorontsov2017orthogonality}
E.~Vorontsov, C.~Trabelsi, S.~Kadoury, and C.~Pal, ``On orthogonality and learning recurrent networks with long term dependencies,'' in \emph{International Conference on Machine Learning}.\hskip 1em plus 0.5em minus 0.4em\relax PMLR, 2017, pp. 3570--3578.

\bibitem{huang2018orthogonal}
L.~Huang, X.~Liu, B.~Lang, A.~Yu, Y.~Wang, and B.~Li, ``Orthogonal weight normalization: Solution to optimization over multiple dependent stiefel manifolds in deep neural networks,'' in \emph{Proceedings of the AAAI Conference on Artificial Intelligence}, vol.~32, no.~1, 2018.

\bibitem{eryilmaz2022understanding}
S.~B. Eryilmaz and A.~Dundar, ``Understanding how orthogonality of parameters improves quantization of neural networks,'' \emph{IEEE Transactions on Neural Networks and Learning Systems}, 2022.

\bibitem{tsitsiklis1986distributed}
J.~Tsitsiklis, D.~Bertsekas, and M.~Athans, ``Distributed asynchronous deterministic and stochastic gradient optimization algorithms,'' \emph{IEEE transactions on automatic control}, vol.~31, no.~9, pp. 803--812, 1986.

\bibitem{chen2021distributed}
S.~Chen, A.~Garcia, and S.~Shahrampour, ``On distributed nonconvex optimization: Projected subgradient method for weakly convex problems in networks,'' \emph{IEEE Transactions on Automatic Control}, vol.~67, no.~2, pp. 662--675, 2021.

\bibitem{shi2015extra}
W.~Shi, Q.~Ling, G.~Wu, and W.~Yin, ``Extra: An exact first-order algorithm for decentralized consensus optimization,'' \emph{SIAM Journal on Optimization}, vol.~25, no.~2, pp. 944--966, 2015.

\bibitem{di2016next}
P.~Di~Lorenzo and G.~Scutari, ``Next: In-network nonconvex optimization,'' \emph{IEEE Transactions on Signal and Information Processing over Networks}, vol.~2, no.~2, pp. 120--136, 2016.

\bibitem{sun2022centralized}
Y.~Sun, M.~Fazlyab, and S.~Shahrampour, ``On centralized and distributed mirror descent: Convergence analysis using quadratic constraints,'' \emph{IEEE Transactions on Automatic Control}, vol.~68, no.~5, pp. 3139--3146, 2022.

\bibitem{chen2023local}
S.~Chen, A.~Garcia, M.~Hong, and S.~Shahrampour, ``On the local linear rate of consensus on the stiefel manifold,'' \emph{IEEE Transactions on Automatic Control}, 2023.

\bibitem{wang2022decentralized}
L.~Wang and X.~Liu, ``Decentralized optimization over the stiefel manifold by an approximate augmented lagrangian function,'' \emph{IEEE Transactions on Signal Processing}, vol.~70, pp. 3029--3041, 2022.

\bibitem{ablin2022fast}
P.~Ablin and G.~Peyr{\'e}, ``Fast and accurate optimization on the orthogonal manifold without retraction,'' in \emph{International Conference on Artificial Intelligence and Statistics}.\hskip 1em plus 0.5em minus 0.4em\relax PMLR, 2022, pp. 5636--5657.

\bibitem{el2016design}
M.~El~Chamie, J.~Liu, and T.~Ba{\c{s}}ar, ``Design and analysis of distributed averaging with quantized communication,'' \emph{IEEE Transactions on Automatic Control}, vol.~61, no.~12, pp. 3870--3884, 2016.

\bibitem{reisizadeh2019exact}
A.~Reisizadeh, A.~Mokhtari, H.~Hassani, and R.~Pedarsani, ``An exact quantized decentralized gradient descent algorithm,'' \emph{IEEE Transactions on Signal Processing}, vol.~67, no.~19, pp. 4934--4947, 2019.

\bibitem{taheri2020quantized}
H.~Taheri, A.~Mokhtari, H.~Hassani, and R.~Pedarsani, ``Quantized decentralized stochastic learning over directed graphs,'' in \emph{International Conference on Machine Learning}.\hskip 1em plus 0.5em minus 0.4em\relax PMLR, 2020, pp. 9324--9333.

\bibitem{kovalev2021linearly}
D.~Kovalev, A.~Koloskova, M.~Jaggi, P.~Richtarik, and S.~Stich, ``A linearly convergent algorithm for decentralized optimization: Sending less bits for free!'' in \emph{International Conference on Artificial Intelligence and Statistics}.\hskip 1em plus 0.5em minus 0.4em\relax PMLR, 2021, pp. 4087--4095.

\bibitem{liu2021linear}
X.~Liu and Y.~Li, ``Linear convergent decentralized optimization with compression,'' in \emph{International Conference on Learning Representations}, 2021.

\bibitem{xiong2022quantized}
Y.~Xiong, L.~Wu, K.~You, and L.~Xie, ``Quantized distributed gradient tracking algorithm with linear convergence in directed networks,'' \emph{IEEE Transactions on Automatic Control}, vol.~68, no.~9, pp. 5638--5645, 2022.

\bibitem{absil2008optimization}
P.-A. Absil, R.~Mahony, and R.~Sepulchre, \emph{Optimization algorithms on matrix manifolds}.\hskip 1em plus 0.5em minus 0.4em\relax Princeton University Press, 2008.

\bibitem{edelman1998geometry}
A.~Edelman, T.~A. Arias, and S.~T. Smith, ``The geometry of algorithms with orthogonality constraints,'' \emph{SIAM journal on Matrix Analysis and Applications}, vol.~20, no.~2, pp. 303--353, 1998.

\bibitem{zhu2017riemannian}
X.~Zhu, ``A riemannian conjugate gradient method for optimization on the stiefel manifold,'' \emph{Computational optimization and Applications}, vol.~67, pp. 73--110, 2017.

\bibitem{sato2022riemannian}
H.~Sato, ``Riemannian conjugate gradient methods: General framework and specific algorithms with convergence analyses,'' \emph{SIAM Journal on Optimization}, vol.~32, no.~4, pp. 2690--2717, 2022.

\bibitem{clarke1995proximal}
F.~H. Clarke, R.~J. Stern, and P.~R. Wolenski, ``Proximal smoothness and the lower-c2 property,'' \emph{J. Convex Anal}, vol.~2, no. 1-2, pp. 117--144, 1995.

\bibitem{balashov2021gradient}
M.~Balashov and R.~Kamalov, ``The gradient projection method with armijo’s step size on manifolds,'' \emph{Computational Mathematics and Mathematical Physics}, vol.~61, pp. 1776--1786, 2021.

\bibitem{zhang2016first}
H.~Zhang and S.~Sra, ``First-order methods for geodesically convex optimization,'' in \emph{Conference on learning theory}.\hskip 1em plus 0.5em minus 0.4em\relax PMLR, 2016, pp. 1617--1638.

\bibitem{cardoso1996equivariant}
J.-F. Cardoso and B.~H. Laheld, ``Equivariant adaptive source separation,'' \emph{IEEE Transactions on signal processing}, vol.~44, no.~12, pp. 3017--3030, 1996.

\bibitem{ablin2024infeasible}
P.~Ablin, S.~Vary, B.~Gao, and P.-A. Absil, ``Infeasible deterministic, stochastic, and variance-reduction algorithms for optimization under orthogonality constraints,'' \emph{Journal of Machine Learning Research}, vol.~25, no. 389, pp. 1--38, 2024.

\bibitem{lecun1998mnist}
Y.~LeCun, ``The mnist database of handwritten digits,'' \emph{http://yann. lecun. com/exdb/mnist/}, 1998.

\end{thebibliography}

\end{document}